\documentclass[11pt,a4paper]{amsart}

\usepackage[margin=1in]{geometry}
\usepackage{amsmath,amssymb,mathtools}
\usepackage{microtype}
\usepackage[colorlinks=true,linkcolor=blue,citecolor=blue,urlcolor=blue]{hyperref}

\newtheorem{theorem}{Theorem}[section]
\newtheorem{proposition}[theorem]{Proposition}
\newtheorem{lemma}[theorem]{Lemma}
\theoremstyle{remark}

\newcommand{\Pp}{\mathbb P}

\title[Strongly 2-primitive sets]{The Second Term for Strongly 2-Primitive Sets}
\author[Przemek Chojecki]{Przemek Chojecki$^{*}$\\[-1pt]
  {\small\normalfont\MakeLowercase{\href{https://ulam.ai}{ulam.ai}}}}
\thanks{$^{*}$AI assistance was used in exploring the argument and in writing this text.}
\date{}

\begin{document}

\begin{abstract}
Let $F(n)$ be the largest size of a set $A\subseteq[1,n]$ such that $a\nmid bc$ whenever $a,b,c\in A$ and $a\notin\{b,c\}$, with $b$ and $c$ allowed to coincide.  We prove
\[
 F(n)=\pi(n)+\left(\frac{27}{2}+o(1)\right)\frac{n^{2/3}}{(\log n)^2}.
\]
This determines the second-order constant conjectured by Erd\H{o}s; the upper bound keeps the leading constants in his multiplicative basis, while the lower bound packs scale-separated prime triples by proper edge-colourings.
\end{abstract}

\maketitle

\section{Introduction}

A set of positive integers is called \emph{strongly $2$-primitive} if no member divides the product of two other members, where the latter two need not be distinct.  Thus $A$ is strongly $2$-primitive precisely when
\begin{equation}\label{eq:strong}
 a\nmid bc\qquad(a,b,c\in A,\ a\ne b,\ a\ne c).
\end{equation}
This terminology distinguishes the condition from the more recent convention in which $b$ and $c$ must be distinct; see, for example, \cite{CLP}.  Erd\H{o}s \cite{Erdos1938} proved that the maximum $F(n)$ in the abstract satisfies
\[
 \pi(n)+c_1\frac{n^{2/3}}{(\log n)^2}
 \leq F(n)\leq
 \pi(n)+c_2\frac{n^{2/3}}{(\log n)^2}
\]
for absolute constants $c_1,c_2>0$.  He later asked whether the error term has an asymptotic constant \cite{Erdos1969}; this is also Erd\H{o}s Problem 793 \cite{Bloom793}.

\begin{theorem}\label{thm:main}
As $n\to\infty$,
\[
 F(n)=\pi(n)+\left(\frac{27}{2}+o(1)\right)
 \frac{n^{2/3}}{(\log n)^2}.
\]
\end{theorem}

The upper bound comes from the four classes in Erd\H{o}s's factorisation argument.  Their two nonnegligible secondary terms have constants $9/2$ and $9$.  For the lower bound, we divide the primes near $n^{1/3}$ into logarithmic bins.  Complete bipartite graphs between two lower bins are properly edge-coloured by primes in a third, higher bin.  The resulting prime triples form a linear $3$-uniform hypergraph, and their products can replace all primes used as vertices.  An exact geometric-series calculation gives the matching constant $9(1/2+1)=27/2$.

All logarithms are natural.  We write
\begin{equation}\label{eq:scale}
 y=n^{1/3},\qquad M=\frac{y}{\log n},\qquad
 S=M^2=\frac{n^{2/3}}{(\log n)^2}.
\end{equation}
The prime number theorem and its standard consequences are used throughout.

\section{The upper bound}

We first isolate the combinatorial form of the factorisation argument.  A multiset below always records multiplicity.

\begin{lemma}[Private factor]\label{lem:private}
Let $\mathcal B$ be a set of positive integers such that every $a\in A$ has a chosen factorisation $a=uv$ with $u,v\in\mathcal B$.  If $A$ is strongly $2$-primitive, then $|A|\leq|\mathcal B|$.
\end{lemma}

\begin{proof}
For $a\in A$, let $E_a=\{u,v\}$ be its chosen factor multiset, and write $\mu_a(x)$ for the multiplicity of $x$ in $E_a$.  We claim that $E_a$ contains an $x$ such that
\begin{equation}\label{eq:private-factor}
 \mu_a(x)>\mu_{a'}(x)\qquad\text{for every }a'\in A\setminus\{a\}.
\end{equation}
If the two factors in $E_a$ are distinct and the claim fails, choose $b,c\in A\setminus\{a\}$ whose chosen factorisations contain the first and second factors, respectively.  Then $a\mid bc$, contrary to \eqref{eq:strong}; it is immaterial if $b=c$.  If $E_a=\{x,x\}$ and the claim fails, some $b\ne a$ has $E_b=\{x,x\}$, which gives $b=x^2=a$, again impossible.

Choose one $x$ satisfying \eqref{eq:private-factor} for each $a$.  Two different elements cannot choose the same $x$, since their two strict multiplicity inequalities would contradict each other.  This gives an injection $A\hookrightarrow\mathcal B$.
\end{proof}

We now give the basis and keep its leading constant.  Put
\begin{align*}
 \mathcal B_0&=[1,n^{3/5}],\\
 \mathcal B_1&=\{p\in\Pp:n^{3/5}<p\leq n\},\\
 \mathcal B_2&=\{pq:p,q\in\Pp,\ p,q\leq y\},\\
 \mathcal B_3&=\{qr:q,r\in\Pp,\ y<q\leq n^{2/5},\ r\leq n/q^2\},
\end{align*}
and $\mathcal B=\mathcal B_0\cup\mathcal B_1\cup\mathcal B_2\cup\mathcal B_3$.

\begin{lemma}[Multiplicative basis]\label{lem:basis}
Every integer $m\leq n$ is a product of two members of $\mathcal B$.
\end{lemma}

\begin{proof}
We use first the following elementary observation: every $m\leq X$ can be written $m=uv$, where $v\leq X^{2/3}$ and $u$ is either prime or at most $X^{2/3}$.  Indeed, this is immediate if $m\leq X^{2/3}$.  Otherwise, if $m$ has a prime divisor $p>X^{1/3}$, take $u=p$ and $v=m/p<X^{2/3}$.  If all prime factors are at most $X^{1/3}$, multiply them until their product first exceeds $X^{1/3}$; this product and its complementary divisor are both at most $X^{2/3}$.

If $m\leq n^{9/10}$, apply the observation with $X=n^{9/10}$.  Both factors belong to $\mathcal B_0\cup\mathcal B_1$.  We may therefore suppose that $m>n^{9/10}$.  If $m$ has a prime factor $p>n^{2/5}$, then $p\in\mathcal B_0\cup\mathcal B_1$ and $m/p<n^{3/5}$, so we are done.  Henceforth all prime factors of $m$ are at most $n^{2/5}$.

Count with multiplicity the prime factors of $m$ that exceed $n^{1/5}$.  If there are at most two, one can find a divisor $u\in[n^{2/5},n^{3/5}]$ as follows.  With zero or one such factor, start with their product and multiply factors at most $n^{1/5}$ until first reaching $n^{2/5}$.  With exactly two, say $p\geq q>n^{1/5}$, start with $p$, omit $q$, and again multiply small factors until reaching $n^{2/5}$; the crossing occurs because $m/q>n^{1/2}$.  In either case the last multiplier is at most $n^{1/5}$, so $u\leq n^{3/5}$, while $m/u\leq n^{3/5}$.  Thus both factors lie in $\mathcal B_0$.

It remains to consider at least three prime factors $p\geq q\geq r>n^{1/5}$.  If $q\leq y$, then $qr\in\mathcal B_2$ and $m/(qr)<n^{3/5}$.  If $q>y$, then
\[
 r\leq\frac{n}{pq}\leq\frac{n}{q^2},
\]
so $qr\in\mathcal B_3$, while $m/(qr)<n^{7/15}<n^{3/5}$.  This completes the proof.
\end{proof}

The secondary contribution of $\mathcal B_3$ is the only count that is not immediate.

\begin{lemma}\label{lem:B3}
As $n\to\infty$,
\[
 \sum_{\substack{y<q\leq n^{2/5}\\q\in\Pp}}
 \pi\left(\frac{n}{q^2}\right)=(9+o(1))S.
\]
\end{lemma}

\begin{proof}
Fix $A>1$; for all sufficiently large $n$ we have $Ay<n^{2/5}$.  Uniformly for $y<q\leq Ay$, the prime number theorem gives
\[
 \pi\left(\frac{n}{q^2}\right)
 =(1+o(1))\frac{3n}{q^2\log n}.
\]
Partial summation and the prime number theorem also give
\[
 \sum_{y<q\leq Ay}\frac1{q^2}
 =(1+o(1))\frac{3}{y\log n}\left(1-\frac1A\right).
\]
Consequently the part with $q\leq Ay$ is
\begin{equation}\label{eq:B3-main}
 \left(9\left(1-\frac1A\right)+o(1)\right)S.
\end{equation}
For the tail, the standard estimates $\pi(t)\ll t/\log t$ and
$\sum_{q>z}q^{-2}\ll(z\log z)^{-1}$ imply
\[
 \sum_{Ay<q\leq n^{2/5}}\pi\left(\frac{n}{q^2}\right)
 \ll\frac{n}{\log n}\sum_{q>Ay}\frac1{q^2}
 \ll\frac{S}{A}.
\]
Letting first $n\to\infty$ and then $A\to\infty$ in \eqref{eq:B3-main} proves the result.
\end{proof}

\begin{proposition}\label{prop:upper}
Every strongly $2$-primitive $A\subseteq[1,n]$ satisfies
\[
 |A|\leq\pi(n)+\left(\frac{27}{2}+o(1)\right)S.
\]
\end{proposition}

\begin{proof}
Unique factorisation shows that
\[
 |\mathcal B_2|=\binom{\pi(y)+1}{2}
 =\left(\frac92+o(1)\right)S.
\]
For $q>y$ we have $n/q^2<y<q$, so every member $qr$ of $\mathcal B_3$ has a unique factor $q>y$.  Lemma~\ref{lem:B3} therefore gives
$|\mathcal B_3|=(9+o(1))S$.  Finally,
\[
 |\mathcal B_0|+|\mathcal B_1|
 =\pi(n)+\lfloor n^{3/5}\rfloor-\pi(n^{3/5})
 =\pi(n)+o(S).
\]
Thus $|\mathcal B|\leq\pi(n)+(27/2+o(1))S$.  Lemmas~\ref{lem:basis} and \ref{lem:private} complete the proof.
\end{proof}

\section{The lower bound}

A family $\mathcal H$ of $3$-element sets is \emph{linear} if two distinct members meet in at most one element.  The next observation records exactly why linear prime triples give the strong, rather than merely ordinary, form of $2$-primitivity.

\begin{lemma}[Linear triples]\label{lem:linear}
Let $\mathcal H$ be a linear family of triples of distinct primes, each with product at most $n$.  Suppose that all its vertex primes are at most $n$, and set
\[
 A_{\mathcal H}
 =\{p\leq n:p\in\Pp,\ p\notin V(\mathcal H)\}
 \ \cup\ 
 \left\{\prod_{p\in E}p:E\in\mathcal H\right\}.
\]
Then $A_{\mathcal H}$ is strongly $2$-primitive and
\[
 |A_{\mathcal H}|=\pi(n)-|V(\mathcal H)|+|\mathcal H|.
\]
\end{lemma}

\begin{proof}
Unique factorisation makes the displayed elements distinct.  A retained prime is not a factor of any triple product.  Now let the target be the product associated with $E\in\mathcal H$.  Every other triple meets $E$ in at most one prime, so two distinct other triple products contain at most two of the three prime factors of the target.  If the two other elements coincide, squaring one triple product still supplies at most one of the three distinct prime factors of the target.  Retained primes supply none of them.  Thus the target cannot divide a product of two other members.
\end{proof}

We construct the linear family.  Fix $h>0$, and for $i\in\mathbb Z$ put
\[
 \Delta_i=e^{(i+1)h}-e^{ih}.
\]
Consider all index cells
\begin{equation}\label{eq:cells}
 (i,j)\in\mathbb Z^2,\qquad i\leq j,\qquad i+2j\leq-4,
\end{equation}
and attach to a cell the third index
\begin{equation}\label{eq:k}
 k=-i-j-3.
\end{equation}
The buffer in \eqref{eq:cells} gives
\begin{equation}\label{eq:order}
 k-j=-i-2j-3\geq1,
\end{equation}
so $i\leq j<k$.

The total cell weight can be evaluated exactly.

\begin{lemma}[Cell weight]\label{lem:weight}
For every $h>0$,
\begin{equation}\label{eq:weight}
 \sum_{\substack{i<j\\i+2j\leq-4}}\Delta_i\Delta_j
 +\frac12\sum_{i\leq-2}\Delta_i^2
 =e^{-h}+\frac12e^{-2h}.
\end{equation}
\end{lemma}

\begin{proof}
Put $t=e^h$.  Since $\sum_{i\leq m}\Delta_i=t^{m+1}$, the off-diagonal sum equals
\begin{align*}
 \sum_{j\leq-1}\Delta_j\sum_{i\leq j-1}\Delta_i
 +\sum_{j\geq0}\Delta_j\sum_{i\leq-2j-4}\Delta_i
 &=\frac1{t+1}+t^{-2}.
\end{align*}
The diagonal contribution is
\[
 \frac12\sum_{i\leq-2}\Delta_i^2
 =\frac{t-1}{2t^2(t+1)}.
\]
Adding and simplifying gives $t^{-1}+\tfrac12t^{-2}$.
\end{proof}

Fix for the moment a finite set $\mathcal C$ of cells satisfying \eqref{eq:cells}.  For every index $r$ occurring in these cells or as a third index, define the prime bin
\[
 P_r=P_r(n)=\{p\in\Pp:ye^{rh}<p\leq ye^{(r+1)h}\},
 \qquad m_r=|P_r|.
\]
For this fixed finite collection, the prime number theorem gives, uniformly in $r$,
\begin{equation}\label{eq:bin-PNT}
 m_r=(3+o(1))M\Delta_r.
\end{equation}
If $(i,j)\in\mathcal C$ and $k$ is given by \eqref{eq:k}, then
$\Delta_k\geq e^h\Delta_j\geq e^h\Delta_i$.  Hence \eqref{eq:bin-PNT} implies, for all sufficiently large $n$,
\begin{equation}\label{eq:capacity}
 m_k\geq\max(m_i,m_j).
\end{equation}

For $i<j$, properly edge-colour the complete bipartite graph between $P_i$ and $P_j$ with $\max(m_i,m_j)$ colours, and inject those colours into $P_k$.  For $i=j$, properly edge-colour the complete graph on $P_i$ with at most $m_i$ colours and inject them into $P_k$.  These elementary colourings may be seen directly: if $m=\max(m_i,m_j)$, label the larger side of the bipartite graph bijectively by $\mathbb Z/m\mathbb Z$, label the smaller side injectively by the same group, and colour an edge by the difference of its endpoint labels.  For a complete graph on an odd number $m$ of vertices, label vertices and colours by $\mathbb Z/m\mathbb Z$ and colour $\{a,b\}$ by $(a+b)/2$.  For even $m$, label the vertices by $\{\infty\}\cup\mathbb Z/(m-1)\mathbb Z$, colour $\{\infty,a\}$ by $a$, and colour $\{a,b\}$ by $(a+b)/2$ in $\mathbb Z/(m-1)\mathbb Z$.  Thus every colour class is a matching.

For each coloured edge $\{p,q\}$, with colour prime $r\in P_k$, put the triple $\{p,q,r\}$ into $\mathcal H_n$.  Its primes are distinct by \eqref{eq:order}, and
\begin{equation}\label{eq:product}
 pqr\leq y^3e^{(i+j+k+3)h}=n.
\end{equation}

\begin{lemma}\label{lem:packing}
The family $\mathcal H_n$ is linear.  Moreover,
\begin{equation}\label{eq:H-count}
 \frac{|\mathcal H_n|}{M^2}\longrightarrow
 9\sum_{\substack{(i,j)\in\mathcal C\\i<j}}\Delta_i\Delta_j
 +\frac92\sum_{\substack{(i,i)\in\mathcal C}}\Delta_i^2.
\end{equation}
\end{lemma}

\begin{proof}
Within one cell, the original graph has no repeated edge and each colour class is a matching.  Hence no pair of primes occurs in two triples from that cell.

For different cells, note that every triple has a sorted index multiset $(i,j,k)$ with $i\leq j<k$ and $i+j+k=-3$.  Two distinct cells cannot share two bin indices: those two indices determine the third by their sum, and sorting then determines the cell.  If the two shared primes were in the same bin, both cells would have to be diagonal, and the sum again determines the third index.  Thus triples from different cells cannot share two primes.  This proves linearity.

An off-diagonal cell supplies $m_im_j$ triples, and a diagonal cell supplies $\binom{m_i}{2}$.  Formula \eqref{eq:H-count} now follows from \eqref{eq:bin-PNT}.
\end{proof}

\begin{proposition}\label{prop:lower}
There are strongly $2$-primitive sets $A\subseteq[1,n]$ such that
\[
 |A|\geq\pi(n)+\left(\frac{27}{2}-o(1)\right)S.
\]
\end{proposition}

\begin{proof}
Let $\varepsilon>0$.  By Lemma~\ref{lem:weight}, first choose $h>0$ sufficiently small and then a finite set $\mathcal C$ of cells such that
\begin{equation}\label{eq:truncate}
 9\sum_{\substack{(i,j)\in\mathcal C\\i<j}}\Delta_i\Delta_j
 +\frac92\sum_{\substack{(i,i)\in\mathcal C}}\Delta_i^2
 >\frac{27}{2}-\varepsilon.
\end{equation}
This is possible because the positive series in \eqref{eq:weight} converges and its value tends to $3/2$ as $h\to0$.

Construct $\mathcal H_n$ from $\mathcal C$ as above.  Since only finitely many bins occur, all their primes are below $n$ for large $n$, and
\[
 |V(\mathcal H_n)|=O_{h,\mathcal C}(M)=o(S).
\]
Equations \eqref{eq:H-count} and \eqref{eq:truncate}, followed by Lemma~\ref{lem:linear}, give
\[
 |A_{\mathcal H_n}|
 =\pi(n)-|V(\mathcal H_n)|+|\mathcal H_n|
 \geq\pi(n)+\left(\frac{27}{2}-\varepsilon-o(1)\right)S.
\]
As $\varepsilon$ is arbitrary, the proposition follows.
\end{proof}

Propositions~\ref{prop:upper} and \ref{prop:lower} prove Theorem~\ref{thm:main}.

\end{document}